\documentclass[12pt]{amsart}
\usepackage{a4wide,amsfonts,graphicx,
amssymb,stmaryrd,amscd,amsmath,latexsym,amsbsy,xypic}
\xyoption{all}
\usepackage{marginnote}

\numberwithin{equation}{section}
\theoremstyle{plain}

\newtheorem{thm}{Theorem}[section]

\newtheorem{prop}[thm]{Proposition}

\newtheorem{defi}[thm]{Definition}
\newtheorem{lem}[thm]{Lemma}
\newtheorem{cor}[thm]{Corollary}

\newtheorem{eg}[thm]{Example}

\theoremstyle{remark}
\newtheorem{rema}[thm]{Remark}

\title{Generalized Onsager algebras}
\author{Jasper V. Stokman}
\address{KdV Institute for Mathematics, University of Amsterdam,
Science Park 105-107, 1098 XG Amsterdam, The Netherlands.}
\email{j.v.stokman@uva.nl}
\begin{document}
%\vspace{2em}
%%%%%%%%%%%%%%%%%%%%%%%%%%%%%%%
%%%%%%%%%%%%%%%%%%%%%%%%%%%%%%%
\maketitle
\begin{abstract}
Let $\mathfrak{g}(A)$ be the Kac-Moody algebra with respect to a symmetrizable generalized Cartan matrix $A$. We give an explicit presentation of the fix-point Lie subalgebra $\mathfrak{k}(A)$ of 
$\mathfrak{g}(A)$ with respect to the Chevalley involution. It is a presentation of $\mathfrak{k}(A)$
involving inhomogeneous versions of the Serre relations, or, from a different perspective, a presentation generalizing the Dolan-Grady presentation of the Onsager algebra. In the finite and untwisted affine case we explicitly compute the structure constants of $\mathfrak{k}(A)$ in terms of a Chevalley type basis of $\mathfrak{k}(A)$. For the symplectic Lie algebra and its untwisted affine extension we explicitly describe the one-dimensional representations of $\mathfrak{k}(A)$.
\end{abstract}
%\tableofcontents
\setcounter{tocdepth}{2}
%%%%%%%%%%%%%%%%%%%%%%%%%%%%%%%
\section{Introduction}
The Onsager algebra is the infinite dimensional complex Lie algebra with linear basis $\{A_k,G_m\}_{k\in\mathbb{Z}, m\in\mathbb{Z}_{>0}}$ and Lie bracket 
\begin{equation}\label{Onsagerstructure}
[A_k,A_\ell]=G_{\ell-k},\qquad [G_m,G_n]=0,\qquad [G_m,A_k]=2(A_{k+m}-A_{k-m}),
\end{equation}
for $k,\ell\in\mathbb{Z}$ and $m,n\in\mathbb{Z}_{>0}$, where $G_{-m}:=-G_m$ ($m>0$)
and $G_0:=0$. It first appeared in Onsager's \cite{O} paper on the two-dimensional Ising model in a zero magnetic field.
It is generated as Lie algebra by $B_0:=-A_{-1}$ and $B_1:=A_0$. The corresponding defining relations, known as the Dolan-Grady \cite{DG} relations, are given by
\[
[B_0,[B_0,[B_0,B_1]]]=-4[B_0,B_1],\qquad
[B_1,[B_1,[B_1,B_0]]]=-4[B_1,B_0].
\]
The Onsager algebra plays an important role in integrable systems, representation theory and special functions, see, e.g., \cite{O,P,DG,D1,D2,Ro,K,HT} and references therein. An important reason for its appearance in many different contexts is the fact that the Onsager algebra is isomorphic to the fix-point Lie subalgebra of the affine Lie algebra $\widetilde{\mathfrak{sl}}_2$ with respect to the Chevalley involution, a result which is due to Roan \cite[Prop. 1]{Ro}.

In this paper we define a generalization $\mathcal{L}(A)$  of the Onsager algebra
in terms of explicit generators and relations depending on a symmetrizable generalized Cartan matrix $A$. It reduces
to the Onsager algebra in Dolan-Grady form
when $A$ is of type $A_1^{(1)}$.

We show that
 $\mathcal{L}(A)$ is isomorphic to the fix-point Lie subalgebra of the Kac-Moody algebra $\mathfrak{g}(A)$ associated to $A$ with respect to its Chevalley involution.
For $A$ of type $A_1^{(1)}$ this is the earlier mentioned result of 
Roan \cite[Prop. 1]{Ro}.
For $A$ of type $A_n^{(1)}$ (resp. $D_n^{(1)}$) we recover the results
of Uglov and Ivanov \cite{UI} (resp. Date and Usami \cite{DU}). 
Their techniques partly
 rely on the loop presentation of the associated affine Lie algebra, which is specific to the affine case. 
 
Much work has recently been done on quantum group analogues of generalized Onsager algebras (see, e.g., \cite{BB,K2} and references therein). Baseilhac and Belliard \cite{BB} introduced generalized $q$-Onsager algebras for $A$ of affine type and provided algebra homomorphisms mapping the generalized $q$-Onsager algebras to the associated quantum affine algebra $U_q(\mathfrak{g}(A))$. 
Kolb \cite{K2} puts these results in the more general framework of quantum symmetric Kac-Moody pairs, which concern quantum analogues of fix-point Kac-Moody subalgebras with respect to involutive automorphisms of the second kind. In this general context Kolb \cite{K2} gave a detailed study of the algebraic structures of the associated fix-point Kac-Moody Lie subalgebras and their quantum analogs, leading in particular cases to explicit identifications with Baseilhac's and Belliard's \cite{BB} generalized $q$-Onsager algebras. The presentations \cite[\S 7]{K2} and \cite[\S 3]{BK} of Kolb's coideal subalgebras in the general context are rather intricate and only explicit when the off-diagonal Cartan integers are $\geq -3$ (for $A$ of finite type, it goes back to the work of Letzter \cite[\S 7]{Le2}). I hope that the present paper will
be a useful step towards a complete and explicit algebraic presentation of Kolb's \cite{K2} coideal subalgebras. 

The content of the paper is as follows.
In Section \ref{OnsagerSection} we introduce the generalized Onsager algebra $\mathcal{L}(A)$
associated to 
a symmetrizable generalized Cartan matrix $A$. We prove that it is isomorphic to the 
fix-point Lie subalgebra $\mathfrak{k}(A)$ of $\mathfrak{g}(A)$ with respect to the Chevalley involution. We define a filtration on $\mathcal{L}(A)$ such that the
associated graded Lie algebra is isomorphic to the graded nilpotent Lie subalgebra of 
$\mathfrak{g}(A)$ generated by the positive root vectors. Finally, we describe the one-dimensional representations of $\mathcal{L}(A)$.

In Section \ref{section3} we consider the special case of finite and untwisted affine indecomposable Cartan matrices, corresponding to $\mathfrak{g}(A)$ being a simple Lie algebra and an untwisted affine Lie algebra respectively. We give in these cases an integral form $\mathfrak{k}_{\mathbb{Z}}(A)$ of the fix-point Lie subalgebra $\mathfrak{k}(A)$ and explicitly describe the structure constants in terms of a Chevalley-type basis of $\mathfrak{k}_{\mathbb{Z}}(A)$.
This leads to a generalization of Onsager's original presentation \cite{O} of the Onsager algebra. 
In Section \ref{section4} we further restrict attention to $A$ of type $C_r$ and type $C_r^{(1)}$
 ($r\geq 1$). For type $C_r$ the fix-point Lie subalgebra $\mathfrak{k}(A)$ is isomorphic
 to $\mathfrak{gl}_r(\mathbb{C})$. In this case we make a detailed comparison of the Dolan-Grady type presentation of $\mathfrak{k}(A)$ and its Serre presentation. Both in the finite and affine case we explicitly describe the one-dimensional representations of $\mathfrak{k}(A)$. In the finite case this will play a role in the upcoming paper of the author and Reshetikhin \cite{RS} on vector-valued Harish-Chandra series.\\

\noindent
{\bf Acknowledgments.} I thank Stefan Kolb, Pascal Baseilhac, Samuel Belliard and Nicolai
Reshetikhin for useful discussions on the topic of the paper. Shortly after the appearance on the arXiv of this paper Xinhong Chen, Ming Lu and Weiqiang Wang informed me that they have obtained an explicit presentation of Kolb's coideal subalgebra associated to a quasi-split Kac-Moody quantum symmetric pair using different techniques. Their results have appeared now in the preprint \cite{CLW}.

%%%%%%%%%%%%%%%%%%%%%%%%%%%%%%%%%%%%
\section{The generalized Onsager algebra}\label{OnsagerSection}
Let $A=(a_{ij})_{i,j=1}^n$ be 
a symmetrizable generalized Cartan matrix and
\[
(\mathfrak{h}(A),\Pi=\{\alpha_i\}_{i=1}^n,\Pi^\vee=\{h_i\}_{i=1}^n)\]
a realization of $A$. So $\mathfrak{h}(A)$ is a complex vector space of dimension
$2n-\textup{rk}(A)$ and $\Pi\subset\mathfrak{h}^*$, $\Pi^\vee\subset\mathfrak{h}$ are linear independent subsets such that $\alpha_j(h_i)=a_{ij}$. The Kac-Moody algebra $\mathfrak{g}(A)$ 
associated to $A$ is the complex Lie algebra generated by $\mathfrak{h}(A)$ and the Chevalley generators $e_i,f_i$ ($i=1,\ldots,n$), with defining relations
\begin{equation*}
\begin{split}
[h,h^\prime]&=0,\\
[h,e_j]&=\alpha_j(h)e_j,\\
 [h,f_j]&=-\alpha_j(h)f_j,\\
[e_i,f_j]&=\delta_{i,j}h_i,\\
(\textup{ad}\,e_i)^{1-a_{ij}}e_j&=0=(\textup{ad}\,f_i)^{1-a_{ij}}f_j,\qquad i\not=j
\end{split}
\end{equation*}
for  $h,h^\prime\in\mathfrak{h}(A)$ and $1\leq i,j\leq n$. Here we use the definition of the Kac-Moody algebra involving the 
Serre relations
\begin{equation}\label{SerreRelations}
(\textup{ad}\,e_i)^{1-a_{ij}}e_j=0=(\textup{ad}\,f_i)^{1-a_{ij}}f_j,\qquad i\not=j,
\end{equation}
which has been shown in \cite{GK} to be equivalent to the usual definition \cite{Ka} when $A$ is symmetrizable.

We recall here some basic properties of Kac-Moody algebras, see \cite{Ka} for further details. 
Let $\omega$ be the Chevalley involution of $\mathfrak{g}(A)$. It is the involutive automorphism of $\mathfrak{g}(A)$ satisfying $\omega|_{\mathfrak{h}(A)}=-\textup{Id}_{\mathfrak{h}(A)}$ and
$\omega(e_i)=-f_i$ ($1\leq i\leq n$). 
Let $\mathfrak{n}(A)\subseteq\mathfrak{g}(A)$
be the Lie subalgebra of $\mathfrak{g}(A)$ generated by $e_1,\ldots,e_n$. It is a graded nilpotent Lie algebra with the generators $e_i$ having degree one. The defining relations of $\mathfrak{n}(A)$
in terms of the Chevalley generators $e_i$ ($i=1,\ldots,n$) are the Serre relations
$(\textup{ad}\,e_i)^{1-a_{ij}}e_j=0$ ($i\not=j$),  
see \cite{GK}. The triangular decomposition of $\mathfrak{g}(A)$ is 
\[
\mathfrak{g}(A)=\omega(\mathfrak{n}(A))\oplus\mathfrak{h}(A)\oplus\mathfrak{n}(A).
\]
Let $Q:=\bigoplus_{i=1}^n\mathbb{Z}\alpha_i$ be the root lattice of $\mathfrak{g}(A)$. 
For $\alpha\in Q$ set
\[
\mathfrak{g}_\alpha(A):=\{x\in\mathfrak{g}(A) \,\, | \,\, [h,x]=\alpha(h)x\quad \forall\, h\in\mathfrak{h}(A)\}.
\]
The $\mathfrak{g}_\alpha(A)$ are finite dimensional and $\mathfrak{g}_0(A)=\mathfrak{h}(A)$. Write $m(\alpha)$ for the dimension of $\mathfrak{g}_\alpha(A)$. 
Note that $m(\alpha_i)=1$ for $i=1,\ldots,n$. 

The root system of $\mathfrak{g}(A)$ is 
\[
\Phi:=\{\alpha\in Q\setminus\{0\}\,\, | \,\, 
m(\alpha)>0\}.
\] 
It decomposes in positive and negative roots,
\[
\Phi=\Phi_+\cup\Phi_-,\qquad \Phi_{\pm}:=\Phi\cap (\pm Q_+),
\]
where $Q_+:=\bigoplus_{i=1}^n\mathbb{Z}_{\geq 0}\alpha_i$.  The height of a positive root
$\alpha\in\Phi_+$ is defined by
\[
\textup{ht}(\alpha):=\sum_{i=1}^nr_i(\alpha)
\]
with $r_i(\alpha)$ the nonnegative integers such that $\alpha=\sum_{i=1}^nr_i(\alpha)\alpha_i$. The nilpotent
Lie algebras $\mathfrak{n}(A)$ and $\omega(\mathfrak{n}(A))$ decompose as
\[
\mathfrak{n}(A)=\bigoplus_{\alpha\in\Phi_+}\mathfrak{g}_\alpha(A),\qquad
\omega(\mathfrak{n}(A))=\bigoplus_{\alpha\in\Phi_-}\mathfrak{g}_\alpha(A).
\]
The $k$th graded component $\mathfrak{n}_k(A)$ of $\mathfrak{n}(A)$ is 
\[
\mathfrak{n}_k(A)=
\bigoplus_{\alpha\in\Phi_+: \textup{ht}(\alpha)=k}\mathfrak{g}_\alpha(A).
\]

%%%%%%%%%%%%%%%%%%%%%%%%%%
\begin{defi}
Write $\mathfrak{k}(A)$ for the fix-point Lie subalgebra of $\mathfrak{g}(A)$ with respect to
the Chevalley involution $\omega$,
\begin{equation}\label{fixpointkA}
\mathfrak{k}(A):=\{x\in\mathfrak{g}(A) \,\, | \,\, \omega(x)=x\}.
\end{equation}
\end{defi}
%%%%%%%%%%%%%%%%%%%%%%%%%
For $\alpha\in\Phi_+$ the subspace
$\mathfrak{g}_\alpha(A)\oplus\mathfrak{g}_{-\alpha}(A)$ of $\mathfrak{g}(A)$ is
$\omega$-stable. Write $(\mathfrak{g}_\alpha(A)\bigoplus\mathfrak{g}_{-\alpha}(A))^\omega$
for its subspace of $\omega$-invariant elements.
Choose for $\alpha\in\Phi_+$ a linear basis $\{e_\alpha^{(j)}\}_{j=1}^{m(\alpha)}$
of $\mathfrak{g}_\alpha(A)$ and set
\[
y_\alpha^{(j)}:=e_{\alpha}^{(j)}+\omega(e_\alpha^{(j)})\in\mathfrak{k}(A).
\]
Then $\{y_\alpha^{(j)}\}_{j=1}^{m(\alpha)}$ is a linear basis of
$(\mathfrak{g}_\alpha(A)\oplus\mathfrak{g}_{-\alpha}(A))^{\omega}$ and 
\[
\mathfrak{k}(A)=\bigoplus_{\alpha\in\Phi_+}
(\mathfrak{g}_\alpha(A)\oplus\mathfrak{g}_{-\alpha}(A))^{\omega}.
\]

Consider the elements
\begin{equation}\label{yi}
Y_i:=e_i+\omega(e_i)=e_i-f_i,\qquad  1\leq i\leq n
\end{equation}
in $(\mathfrak{g}_{\alpha_i}(A)\oplus\mathfrak{g}_{-\alpha_i}(A))^{\omega}\subseteq\mathfrak{k}(A)$.
As a special case of \cite[Lem. 2.7]{K2} we have
%%%%%%%%%%%%%%%%%%%%%%%%%%%%%%%%%%%%%%%%%%%%
\begin{lem}\label{generatinglemma}
The elements $Y_1,\ldots,Y_n$ generate the Lie algebra $\mathfrak{k}(A)$.
\end{lem}
%%%%%%%%%%%%%%%%%%%%%%%%%%%%%%%%%%%%%%%%%%%%
\begin{proof}
For completeness we recall the proof.
Note that for $\alpha=\alpha_i\in\Pi$ a simple root we have $m(\alpha_i)=1$ and $y_{\alpha_i}^{(1)}$ is a
nonzero constant multiple of $Y_i$.
We now prove by induction to the height of $\alpha\in\Phi_+$  that $y_\alpha^{(j)}$
lies in the Lie subalgebra $\mathfrak{k}^\prime(A)$ of $\mathfrak{k}(A)$ generated by $Y_1,\ldots,Y_n$.

Suppose $\alpha\in\Phi_+$ has height $k>1$ and fix $1\leq j\leq m(\alpha)$. Since $\mathfrak{n}(A)$ is generated by $e_1,\ldots,e_n$, we can write
\[
e_\alpha^{(j)}=\sum_{\mathbf{i}}c_{\mathbf{i}}[e_{i_1},\ldots [e_{i_{k-2}},[e_{i_{k-1}},e_{i_k}]]\ldots]
\]
for certain constants $c_{\mathbf{i}}\in\mathbb{C}$, with the sum over
$k$-tuples $\mathbf{i}=(i_1,\ldots,i_k)$ such that $\alpha=\alpha_{i_1}+\cdots+\alpha_{i_k}$. Then
\[
y_\alpha^{(j)}-\widetilde{y}_\alpha^{(j)}\in\bigoplus_{\beta\in\Phi_+: \textup{ht}(\beta)<k}
(\mathfrak{g}_\beta(A)\oplus\mathfrak{g}_{-\beta}(A))^\omega
\]
for the element 
\[
\widetilde{y}_\alpha^{(j)}:=\sum_{\mathbf{i}}c_{\mathbf{i}}[Y_{i_1},\ldots [Y_{i_{k-2}},
[Y_{i_{k-1}},Y_{i_k}]]\ldots]\in\mathfrak{k}^\prime(A).
\]
By the induction hypothesis we conclude that 
$y_\alpha^{(j)}-\widetilde{y}_\alpha^{(j)}\in\mathfrak{k}^\prime(A)$, hence
$y_\alpha^{(j)}\in\mathfrak{k}^\prime(A)$.
\end{proof}
%%%%%%%%%%%%%%%%%%%%%%%%%%%%%%%%%%%%%%%%%%%%
The defining relations of $\mathfrak{k}(A)$ in terms of the generators $Y_1,\ldots,Y_n$ take the form of inhomogeneous Serre relations involving the following integers $c_{s}^{ij}[r]$ for $1\leq i\not=j\leq n$ and $r\geq s\geq 0$.

%%%%%%%%%%%%%%%%%%%%%%%%%%
\begin{defi} 
Let $1\leq i\not=j\leq n$. 
Set $c_r^{ij}[r]:=1$ ($r\geq 0$) and $c_{r-1}^{ij}[r]:=0$ ($r\geq 1$). For $r\geq 2$, define $c_s^{ij}[r]$ for 
$r\geq s\geq 0$
recursively by
\begin{equation}\label{recursion}
c_s^{ij}[r]=c_{s-1}^{ij}[r-1]-(r-1)(r-2+a_{ij})c_s^{ij}[r-2],
\end{equation} 
with the convention that $c_{-1}^{ij}[r]:=0$. 
\end{defi}
%%%%%%%%%%%%%%%%%%%%%%%%%
Note that $c_s^{ij}[r]$ only depends on the matrix coefficient $a_{ij}$ of the generalized Cartan matrix $A$.
Note furthermore that $c_0^{ij}[2\ell+1]=0$ ($\ell\in\mathbb{Z}_{\geq 0}$) and
\begin{equation}\label{even}
c_0^{ij}[2\ell]=(-1)^\ell\prod_{k=1}^{\ell}(2k-1)(2k-2+a_{ij}),\qquad \ell\in\mathbb{Z}_{\geq 0}.
\end{equation}
The list of integers $c_s^{ij}[r]$ for $0\leq s\leq r\leq 5$ is explicitly given by
\begin{equation}\label{integer5}
\begin{split}
c_0^{ij}[0]&=1,\\
(c_0^{ij}[1],c_1^{ij}[1])&=(0,1),\\
(c_0^{ij}[2], c_1^{ij}[2], c_2^{ij}[2])&=(-a_{ij},0,1),\\
(c_0^{ij}[3], c_1^{ij}[3], c_2^{ij}[3], c_3^{ij}[3])&=(0,-3a_{ij}-2,0,1),\\
(c_0^{ij}[4], c_1^{ij}[4], c_2^{ij}[4], c_3^{ij}[4], c_4^{ij}[4])&=(3a_{ij}^2+6a_{ij},0,-6a_{ij}-8,0,1),\\
(c_0^{ij}[5], c_1^{ij}[5], c_2^{ij}[5], c_3^{ij}[5], c_4^{ij}[5], c_5^{ij}[5])&=(0,15a_{ij}^2+50a_{ij}+24,0,-10a_{ij}-20,0,1).
\end{split}
\end{equation}
%%%%%%%%%%%%%%%%%%%%%%%%%%%%%%%%%%%
\begin{prop}
In $\mathfrak{g}(A)$ we have
\begin{equation}\label{almostrelation}
\sum_{s=0}^rc_s^{ij}[r]\,(\textup{ad}\,Y_i)^sY_j=
(\textup{ad}\,e_i)^re_j+(-1)^{r-1}(\textup{ad}\,f_i)^rf_j,\qquad 1\leq i\not=j\leq n
\end{equation}
for $r\geq 0$.
\end{prop}
%%%%%%%%%%%%%%%%%%%%%%%%%%%%%%%%%%%
\begin{proof} 
Fix $i\not=j$.
Induction to $r$ shows that 
\[
(\textup{ad}\,e_i)(\textup{ad}f_i)^rf_j=-r(r-1+a_{ij})(\textup{ad}\,f_i)^{r-1}f_j
\]
for $r\geq 0$, where the right hand side is read as zero for $r=0$. Applying the Chevalley involution $\omega$ shows that the formula with $e_i\leftrightarrow f_i$ and $e_j\leftrightarrow f_j$ also holds true. 

We now proceed to prove \eqref{almostrelation} by induction to $r$, the statement being trivial for $r=0$ and $r=1$. If \eqref{almostrelation} holds true up to and including $r\in\mathbb{Z}_{>0}$ then acting by $\textup{ad}\,Y_i=\textup{ad}(e_i-f_i)$ on both sides of \eqref{almostrelation} gives, in view of the previous paragraph,
\begin{equation*}
\begin{split}
\sum_{s=0}^{r+1}c_{s-1}^{ij}[r](\textup{ad}\,Y_i)^sY_j&=
(\textup{ad}\,e_i)^{r+1}e_j+(-1)^{r}(\textup{ad}\,f_i)^{r+1}f_j\\
&+r(r-1+a_{ij})\bigl((\textup{ad}\,e_i)^{r-1}e_j+(-1)^{r-2}(\textup{ad}\,f_i)^{r-1}f_j\bigr).
\end{split}
\end{equation*}
Applying the induction hypothesis to the second line 
and moving the resulting term to the other side of the equation
establishes the induction step, since
\[
c_s^{ij}[r+1]=c_{s-1}^{ij}[r]-r(r-1+a_{ij})c_s^{ij}[r-1],\qquad 0\leq s<r,
\]
$c_r^{ij}[r+1]=0=c_{r-1}^{ij}[r]$ and $c_{r+1}^{ij}[r+1]=1=c_r^{ij}[r]$.
\end{proof}
%%%%%%%%%%%%%%%%%%%%%%%%%%%%%%%%%
For $r=1-a_{ij}$ the identity \eqref{almostrelation} reduces to
\begin{equation}\label{DGrelationsg(A)}
\sum_{s=0}^{1-a_{ij}}c_s^{ij}[1-a_{ij}](\textup{ad}\,Y_i)^sY_j=0,\qquad
1\leq i\not=j\leq n
\end{equation}
due to the Serre relations \eqref{SerreRelations} in $\mathfrak{g}(A)$. We will show that these
are the defining relations of $\mathfrak{k}(A)$ in terms of the generators $y_1,\ldots,y_n$ of $\mathfrak{k}(A)$. In other words, we will show that $\mathfrak{k}(A)$ is isomorphic to 
the {\it generalized Onsager algebra}, which we define as follows:
%%%%%%%%%%%%%%%%%%%%%%%%%
\begin{defi}[Generalized Onsager algebra]\label{DefgenOns}
Let $A=(a_{ij})_{i,j=1}^n$ be a symmetrizable generalized Cartan matrix. The generalized Onsager algebra
$\mathcal{L}(A)$ is the complex Lie algebra with generators $B_1,\ldots,B_n$ and defining relations
the inhomogeneous Serre relations
\begin{equation}\label{DGrelations}
\sum_{s=0}^{1-a_{ij}}c_s^{ij}[1-a_{ij}]\,(\textup{ad}\,B_i)^sB_j=0\qquad 1\leq i\not=j\leq n.
\end{equation}
\end{defi}
%%%%%%%%%%%%%%%%%%%%%%%%%%%%%%%%%%%%%%%%%%%%
By \eqref{integer5}, the inhomogeneous Serre relations \eqref{DGrelations}
for $a_{ij}\geq -4$ are given by the following concrete list:
\begin{equation}\label{integer5two}
\begin{split}
[B_i,B_j]&=0,\qquad\qquad\qquad\qquad\qquad\quad\,\,\qquad\quad\,\,\,\, \hbox{ if } a_{ij}=0,\\
[B_i,[B_i,B_j]]&=-B_j,\qquad\qquad\qquad\qquad\qquad\qquad\quad\,\,\,\,\hbox{ if } a_{ij}=-1,\\
[B_i,[B_i,[B_i,B_j]]]&=-4[B_i,B_j],\qquad\quad\qquad\qquad\,\,\,\qquad\quad\,\,\, \hbox{ if } a_{ij}=-2,\\
[B_i,[B_i,[B_i,[B_i,B_j]]]]&=-10[B_i,[B_i,B_j]]-9B_j,\qquad\qquad\quad\,\,\,\, \hbox{ if } a_{ij}=-3,\\
[B_i,[B_i,[B_i,[B_i,[B_i,B_j]]]]]&=-20[B_i,[B_i,[B_i,B_j]]]-64[B_i,B_j], \quad \hbox{ if } a_{ij}=-4.
\end{split}
\end{equation}
\begin{rema}\label{DGrema}
For the generalized Cartan matrix
\begin{equation}\label{typeA11}
A=\left(\begin{matrix} 2 & -2\\ -2 & 2\end{matrix}\right)
\end{equation}
of affine type $A_1^{(1)}$, $\mathcal{L}(A)$ is the Onsager \cite{O} algebra and the inhomogeneous Serre relations \eqref{DGrelations} are the Dolan-Grady \cite{DG} relations. 
For $A$ of 
affine type $A_n^{(1)}$ (respectively $D_n^{(1)}$), the generalized Onsager algebra was introduced by Uglov and Ivanov \cite{UI} (respectively Date and Usami \cite{DU}).  For $A$ of arbitrary affine type, generalized Onsager algebras and their quantum analogs have been introduced by Baseilhac and Belliard \cite{BB}. 
\end{rema}

We turn $\mathcal{L}(A)$ into a filtered Lie algebra with filtration $\mathcal{L}(A)=\bigcup_{j=1}^{\infty}\mathcal{L}_j(A)$ given by
\[
\mathcal{L}_j(A):=\textup{span}\{[B_{i_1},\ldots [B_{i_{m-2}},[B_{i_{m-1}},
B_{i_m}]]\ldots] \,\, | \,\, 1\leq i_s\leq n \,\, \& \,\, m\leq j\}.
\]
Write
\[
\textup{Gr}(\mathcal{L}(A))=\bigoplus_{j=1}^{\infty}\textup{Gr}_j(\mathcal{L}(A))
\] 
for the associated graded Lie algebra, with $\textup{Gr}_j(\mathcal{L}(A)):=
\mathcal{L}_j(A)/\mathcal{L}_{j-1}(A)$ and $\mathcal{L}_{0}(A):=\{0\}$. Elements
in $\textup{Gr}_j(\mathcal{L}(A))$  will be denoted by
\[
[x]_j:=x+\mathcal{L}_{j-1}(A)\in\textup{Gr}_j(\mathcal{L}(A)),\qquad x\in\mathcal{L}_j(A).
\]
Note that the inhomogeneous Serre relations
\eqref{DGrelations} in $\mathcal{L}(A)$ turn into the usual Serre relations
\begin{equation}\label{homDGrelations}
(\textup{ad} [B_i]_1)^{1-a_{ij}}[B_j]_1=0,\qquad i\not=j
\end{equation}
for the generators $\{[B_i]_1\}_{i=1}^n$ of the graded Lie algebra $\textup{Gr}(\mathcal{L}(A))$. 
%%%%%%%%%%%%%%%%%%%%%%%%%%%%%%%%%%%%%%%%%
\begin{thm}\label{mainTHM}
Let $A=(a_{ij})_{i,j=1}^n$ be a symmetrizable generalized Cartan matrix.\\
{\bf a.} The Lie algebra homomorphism
\begin{equation}\label{psi}
\psi: \mathcal{L}(A)\rightarrow \mathfrak{k}(A)
\end{equation}
defined by $\psi(B_i)=Y_i$ for $i=1,\ldots,n$, is an isomorphism.\\
{\bf b.} The graded Lie algebra homomorphism
\begin{equation}\label{varphi}
\varphi: \mathfrak{n}(A)\rightarrow\textup{Gr}(\mathcal{L}(A))
\end{equation}
defined by $\varphi(e_i)=[B_i]_1$ for $i=1,\ldots,n$, is an isomorphism.
\end{thm}
%%%%%%%%%%%%%%%%%%%%%%%%%%%%%%%%%%%%%
\begin{proof}
By the earlier remarks it is clear that $\psi$ and $\varphi$ are well defined
and surjective.\\
{\bf a.} 
We show that $\psi$ is injective. Let $y\in \textup{Ker}(\psi)\subseteq\mathcal{L}(A)$ and let $\ell\geq 1$ such that $y\in\mathcal{L}_\ell(A)$. Write
\begin{equation}\label{uptoell1}
y=\sum_{k=1}^{\ell}\Bigl(\sum_{i_1,\ldots,i_k=1}^nd_{i_1,\ldots,i_k}[B_{i_1},\ldots [B_{i_{k-2}}, [B_{i_{k-1}},
B_{i_k}]]\ldots]\Bigr)
\end{equation}
with coefficients $d_{i_1,\ldots,i_k}\in\mathbb{C}$. We prove that $y=0$ by induction to $\ell$.
If $\ell=1$ then $y=\sum_{i=1}^nd_iB_i$,
and $0=\psi(y)=\sum_{i=1}^nd_i(e_i-f_i)$ in $\mathfrak{k}(A)$ implies $d_i=0$ for all $i$. Suppose $\ell\geq 2$.
Applying $\psi: \mathcal{L}(A)\rightarrow\mathfrak{k}(A)$ to the identity \eqref{uptoell1} then yields 
\begin{equation}\label{halfway}
\sum_{k=1}^{\ell}\Bigl(\sum_{i_1,\ldots,i_k=1}^nd_{i_1,\ldots,i_k}[Y_{i_1},\ldots [Y_{i_{k-2}}, [Y_{i_{k-1}}, Y_{i_k}]]\ldots]\Bigr)=0
\end{equation}
in $\mathfrak{k}(A)\subset\mathfrak{g}(A)$. Consider the direct sum decomposition
\begin{equation}\label{along}
\mathfrak{g}(A):=\bigoplus_{k\in\mathbb{Z}}\mathfrak{g}(k)
\end{equation}
where $\mathfrak{g}(0):=\mathfrak{h}(A)$ and, for $k\in\mathbb{Z}_{>0}$, $\mathfrak{g}(k):=\bigoplus_{\alpha\in\Phi_+: \textup{ht}(\alpha)=k}\mathfrak{g}_\alpha(A)$ and $\mathfrak{g}(-k):=
\omega(\mathfrak{g}(k))$. Choosing the $\mathfrak{g}(\ell)$-component of the identity \eqref{halfway} along the direct sum decomposition \eqref{along}
we obtain
\begin{equation}\label{ell1}
\sum_{i_1,\ldots,i_\ell=1}^nd_{i_1,\ldots,i_\ell}[e_{i_1},\ldots [e_{i_{\ell-2}}, [e_{i_{\ell-1}},
e_{i_\ell}]]\ldots]=0
\end{equation}
in $\mathfrak{n}_\ell(A)$. Applying $\varphi: \mathfrak{n}(A)\rightarrow\textup{Gr}(\mathcal{L}(A))$ to  \eqref{ell1} then implies that
\[
\sum_{i_1,\ldots,i_\ell=1}^nd_{i_1,\ldots,i_\ell}[B_{i_1},\ldots [B_{i_{\ell-2}}, [B_{i_{\ell-1}},
B_{i_\ell}]]\ldots]\in \mathcal{L}_{\ell-1}(A).
\]
Returning to \eqref{uptoell1} we conclude that $y\in\mathcal{L}_{\ell-1}(A)$, hence $y=0$ by the
induction hypothesis. This show that $\psi: \mathcal{L}(A)\rightarrow\mathfrak{k}(A)$ is an isomorphism of Lie algebras.\\
{\bf b.} We show that the surjective graded Lie algebra homomorphism $\varphi: \mathfrak{n}(A)\rightarrow \textup{Gr}(\mathcal{L}(A))$ is injective. Fix $\ell\geq 1$. It suffices to show that
the set 
\[
\mathcal{S}_\ell:=\bigcup_{\alpha\in\Phi_+:\, \textup{ht}(\alpha)=\ell}
\{\varphi(e_\alpha^{(j)}) \,\, | \,\, 
1\leq j\leq m(\alpha)\}
\]
is linear independent in $\textup{Gr}_\ell(\mathcal{L}(A))=\mathcal{L}_\ell(A)/\mathcal{L}_{\ell-1}(A)$.

Identify $\mathcal{L}(A)\simeq\mathfrak{k}(A)$
using the Lie algebra isomorphism $\psi$. Then we have  
\[
\mathcal{L}_k(A)\subseteq \mathfrak{g}_{\leq k}:=\bigoplus_{i=-\infty}^k\mathfrak{g}(i)
\]
for $k\geq 1$. Hence we have a well defined linear map $\pi_\ell: \textup{Gr}_\ell(\mathcal{L}(A))\rightarrow \mathfrak{g}_{\leq \ell}/\mathfrak{g}_{\leq \ell-1}$ defined by
$\pi_\ell([x]_\ell):=x+\mathfrak{g}_{\leq \ell-1}$.
For $\alpha\in\Phi_+$ with $\textup{ht}(\alpha)=\ell$ and $j\in\{1,\ldots,m(\alpha)\}$ one shows, in a similar manner as in the proof of part {\bf a}, that
\[
\pi_\ell(\varphi(e_\alpha^{(j)}))=e_\alpha^{(j)}+\mathfrak{g}_{\leq \ell-1}.
\]
Consequently
\[\pi_\ell(\mathcal{S}_\ell)=\bigcup_{\alpha\in\Phi_+: \textup{ht}(\alpha)=\ell}\{e_\alpha^{(j)} +\mathfrak{g}_{\leq \ell-1}\},
\]
which is a linear independent set in $\mathfrak{g}_{\leq\ell}/\mathfrak{g}_{\leq \ell-1}$. We conclude that $\mathcal{S}_\ell$ is a linear independent set in $\textup{Gr}_{\ell}(\mathcal{L})$. This completes the proof of {\bf b}.
\end{proof}
%%%%%%%%%%%%%%%%%%%%%%%%%%%%%%%%%%%%%%%%
In view of Remark \ref{DGrema} and the previous theorem, we introduce the following terminology.
%%%%%%%%%%%%%%%%%%%%%%%%%%%%%%%%%%%%%%%%%%
\begin{defi}
We call $\mathcal{L}(A)$ the {\it Dolan-Grady type} presentation of the fix-point Lie subalgebra $\mathfrak{k}(A)$.
\end{defi}
%%%%%%%%%%%%%%%%%%%%%%%%%%%%%%%%%%%%%%%%%%%%

We end this section by describing the 
space $\textup{ch}(\mathfrak{k}(A))$ of one-dimensional representations
of $\mathfrak{k}(A)$.
Note that as vector spaces, 
\[
\textup{ch}(\mathfrak{k}(A))\simeq\bigl(\mathfrak{k}(A)/[\mathfrak{k}(A),\mathfrak{k}(A)]\bigr)^*.
\]

Define
\[
\mathcal{E}_A:=\{j\in\{1,\ldots,n\} \,\,\,\, | \,\,\,\, a_{ij}\equiv 0 \,\,\ (\textup{mod } 2)\,\,\,\,\forall\, i\in\{1,\ldots,n\}\}.
\]
%%%%%%%%%%%%%%%%%%%%%%%%%%%%%%%%%%%%%%%%
\begin{prop}\label{onedimrep}
We have a linear isomorphism
\[
\mathbb{C}^{\mathcal{E}_A}\overset{\sim}{\longrightarrow}\textup{ch}(\mathfrak{k}(A)),
\qquad \mathbf{t}\mapsto \chi_{\mathbf{t}},
\]
with $\chi_{\mathbf{t}}\in\textup{ch}(\mathfrak{k}(A))$ ($\mathbf{t}
=(t_j)_{j\in\mathcal{E}_A}\in\mathbb{C}^{\mathcal{E}_A}$) defined by 
\begin{equation*}
\chi_{\mathbf{t}}(Y_j)=
\begin{cases}
t_j\quad &\hbox{ if }\,\, j\in\mathcal{E}_A,\\
0\quad &\hbox{ if }\,\, j\not\in\mathcal{E}_A.
\end{cases}
\end{equation*}
\end{prop}
%%%%%%%%%%%%%%%%%%%%%%%%%%%%%%%%%%%%%%%%%
\begin{proof}
Fix $j\in\{1,\ldots,n\}$.
If $i\not=j$ then $c_0^{ij}[1-a_{ij}]=0$ if $a_{ij}$ is even, and \eqref{even} implies that $c_0^{ij}[1-a_{ij}]\not=0$ if $a_{ij}$ is odd. By the defining inhomogeneous Serre relations \eqref{DGrelations} for $\mathcal{L}(A)$, a one-dimensional representation $\chi\in\textup{ch}(\mathcal{L}(A))$
can thus take any value at $B_j$ if $j\in\mathcal{E}_A$, while it must map $B_j$ to zero if
$j\not\in\mathcal{E}_A$.
The result now immediately follows from Theorem \ref{mainTHM}\,{\bf a}.
\end{proof}
%%%%%%%%%%%%%%%%%%%%%%%%%%%%%%%%%%%%%%%%%%%

%%%%%%%%%%%%%%%%%%%%
\section{Onsager type presentation in the finite and untwisted affine case}\label{section3}
%%%%%%%%%%%%%%%%%%%%%

In this section we describe some additional properties of $\mathfrak{k}(A)$ in case the symmetrizable generalized Cartan matrix $A$ is of finite or untwisted affine type.
In these two cases we discuss integral forms of $\mathfrak{k}(A)$ and explicitly give the
structure constants of $\mathfrak{k}(A)$ with respect to a suitable integral basis of $\mathfrak{k}(A)$.
We show that this leads to Onsager type presentations of $\mathfrak{k}(A)$ if $A$ is of untwisted affine type.
.
%%%%%%%%%%%%%%%%%%%%%%%%%%%%%%%
\subsection{The finite case}
%%%%%%%%%%%%%%%%%%%%%%%%%%%%%%%
Let $A$ be an indecomposable generalized Cartan matrix of finite type (which is automatically symmetrizable by \cite[Lem. 4.6]{Ka}). By
Serre's Theorem \cite[Thm. 18.2]{Hu} the corresponding Kac-Moody Lie algebra $\mathfrak{g}(A)$ is a complex simple Lie algebra and the pair
$(\mathfrak{g}(A),\mathfrak{k}(A))$ is the complexification of an irreducible real split symmetric pair. Let us discuss this case now from this perspective.

Let $\mathfrak{g}$ be a simple Lie algebra over $\mathbb{C}$ of rank $r$  and fix a Cartan subalgebra 
$\mathfrak{h}\subset\mathfrak{g}$.
Write $\Phi\subset\mathfrak{h}^\ast$ for its root system, $\Phi_+$ for a choice of positive roots,
and $\Pi:=\{\alpha_1,\ldots,\alpha_r\}$ for the corresponding simple roots. 
Let $(\cdot,\cdot)$ be a nondegenerate invariant symmetric bilinear form on $\mathfrak{g}$. It is unique up to a nonzero scalar multiple. It is non-degenerate when restricted to $\mathfrak{h}\times\mathfrak{h}$. For $\lambda\in\mathfrak{h}^*$ let $t_\lambda\in\mathfrak{h}$ be the Cartan element such that $(t_\lambda,h)=\lambda(h)$ for all 
$h\in\mathfrak{h}$. Write $(\lambda,\mu):=\mu(t_\lambda)$ ($\lambda,\mu\in\mathfrak{h}^*$) for the induced bilinear form on $\mathfrak{h}^*$. For later purposes (see Subsection \ref{affineSection}) it is convenient to normalize the form $(\cdot,\cdot)$ on $\mathfrak{g}$ such that $(\alpha,\alpha)=2$ for long roots $\alpha$. Define 
\[
h_\alpha:=\frac{2t_\alpha}{(\alpha,\alpha)},\qquad \alpha\in\Phi,
\]
and write $h_i:=h_{\alpha_i}$
for $i=1,\ldots,r$.
Set $\Pi^\vee:=\{h_i\}_{i=1}^r\subset\mathfrak{h}$. Then 
$A=(a_{ij})_{i,j=1}^r:=(\alpha_j(h_i))_{i,j=1}^r$ is the Cartan matrix of $\mathfrak{g}$,
and $(\mathfrak{h},\Pi,\Pi^\vee)$ is a realization of $A$.

Fix $e_i\in\mathfrak{g}_{\alpha_i}$ and $f_i\in\mathfrak{g}_{-\alpha_i}$ ($1\leq i\leq r$) such that
$[e_i,f_i]=h_i$. Then Serre's Theorem \cite[Thm. 18.2]{Hu} shows that
$\mathfrak{g}\simeq\mathfrak{g}(A)$ by identifying $e_1,\ldots,e_r,f_1,\ldots,f_r$ with the Chevalley
generators of $\mathfrak{g}(A)$.
We will freely use the resulting notations and results on $\mathfrak{g}(A)$ from the previous section,
only dropping the dependence on $A$.
In particular, we write $\mathfrak{g}_\alpha$ for $\mathfrak{g}_\alpha(A)$, $\mathfrak{k}$ for
$\mathfrak{k}(A)$, etc. Note that the natural number $n$ from the previous section equals the rank $r$ of $\mathfrak{g}$. 

Recall that the Chevalley involution $\omega$ is given by 
$\omega|_{\mathfrak{h}}=-\textup{Id}_{\mathfrak{h}}$ and $\omega(e_i)=-f_i$ for $i=1,\ldots,r$. The corresponding generators $Y_i$ ($1\leq i\leq r$) of $\mathfrak{k}$ are $Y_i=e_i-f_i$ ($i=1,\ldots,r$).
Theorem \ref{mainTHM}\,{\bf a} leads to the explicit presentation of $\mathfrak{k}$ in terms of the
generators $Y_i$ ($i=1,\ldots,r$) by identifying $\mathfrak{k}$ with the generalized Onsager algebra $\mathcal{L}$ with respect to the Cartan matrix $A$.

%%%%%%%%%%%%%%%%%%%%%%%%%%%%%%%%%%%%%
\begin{eg}\label{example}
Let $\mathfrak{g}=\mathfrak{sl}_{r+1}(\mathbb{C})$ be the special linear Lie algebra and $E_{i,j}$ ($i,j=1,\ldots,r+1$) the standard matrix units in $\mathfrak{gl}_{r+1}(\mathbb{C})$. Take the diagonal matrices in $\mathfrak{g}$ as the Cartan subalgebra $\mathfrak{h}$ of $\mathfrak{g}$. Then
\[
e_i:=E_{i,i+1},\qquad f_i:=E_{i+1,i},\qquad i=1,\ldots,r
\]
are Chevalley generators of $\mathfrak{g}$.
One has
\[
\omega(X):=-X^T,\qquad X\in\mathfrak{sl}_{r+1}(\mathbb{C})
\] 
for the associated Chevalley involution, with $X^T$ is the transpose of the matrix $X$, and
the associated fix-point Lie subalgebra
\[
\mathfrak{k}=\{X\in\mathfrak{sl}_{r+1}(\mathbb{C}) \,\, | \,\, \omega(X)=X\}
\]
is the orthogonal Lie algebra $\mathfrak{so}_{r+1}(\mathbb{C})$.
\end{eg}
%%%%%%%%%%%%%%%%%%%%%%%%%%%%%%%%%%%%

We extend the subset $\{h_i,e_i,f_i\}_{i=1}^r$ to a 
Chevalley basis $\{h_i,e_\alpha\}_{1\leq i\leq r, \alpha\in\Phi}$ of $\mathfrak{g}$ as follows.
We set
$e_{\alpha_i}:=e_i$ and $e_{-\alpha_i}:=f_i$ for $i=1,\ldots,r$ and choose for the remaining roots $\alpha$ root vectors
$e_\alpha\in\mathfrak{g}_\alpha$ such that $[e_\alpha,e_{-\alpha}]=h_\alpha$ and $\omega(e_\alpha)=-e_{-\alpha}$ (see, e.g., \cite[\S 25.2]{Hu} for a detailed discussion on
the existence of Chevalley bases of $\mathfrak{g}$). 
Then
\[
\mathfrak{g}_{\mathbb{Z}}:=\textup{span}_{\mathbb{Z}}\{e_\alpha,h_i\,\,\, | \,\,\, 
\alpha\in\Phi,\,\, 1\leq i\leq r\}.
\]
is an integral form of $\mathfrak{g}$, in the sense that the Lie bracket $\lbrack\cdot,\cdot\rbrack$ of $\mathfrak{g}$ restricts to a 
Lie bracket $\lbrack\cdot,\cdot\rbrack: \mathfrak{g}_{\mathbb{Z}}\times\mathfrak{g}_{\mathbb{Z}}
\rightarrow\mathfrak{g}_{\mathbb{Z}}$ on $\mathfrak{g}_{\mathbb{Z}}$ and 
$\mathfrak{g}\simeq \mathbb{C}\otimes_{\mathbb{Z}}\mathfrak{g}_{\mathbb{Z}}$.
Note that the Chevalley involution $\omega$ restricts to an involution on $\mathfrak{g}_{\mathbb{Z}}$.

If $\alpha,\beta\in\Phi$ and $\alpha+\beta\in\Phi$ then $[e_\alpha,e_\beta]=\kappa_{\alpha,\beta}e_{\alpha+\beta}$ with coefficients $\kappa_{\alpha,\beta}\in\mathbb{Z}\setminus\{0\}$ satisfying
$\kappa_{\alpha,\beta}=
-\kappa_{-\alpha,-\beta}$. The integers $\kappa_{\alpha,\beta}$ are easily computable up to sign, see 
\cite[\S 25.2]{Hu}. Set $\kappa_{\alpha,\beta}:=0$ if $\alpha+\beta\not\in\Phi$. In what follows
we will simply write
\begin{equation}\label{calphabeta}
[e_\alpha,e_\beta]=\kappa_{\alpha,\beta}e_{\alpha+\beta}\qquad (\alpha,\beta\in\Phi: \,\, 
\alpha+\beta\not=0)
\end{equation}
as identities in $\mathfrak{g}$, which should be read as $[e_\alpha,e_\beta]=0$ if $\alpha+\beta\not\in\Phi$.

Define 
\begin{equation}\label{yalpha}
y_\alpha:=e_\alpha-e_{-\alpha}\in\mathfrak{k},\qquad \alpha\in\Phi.
\end{equation}
Then $y_{-\alpha}=-y_\alpha$ ($\alpha\in\Phi$) and $\{y_\alpha\}_{\alpha\in\Phi_+}$
is a linear basis of $\mathfrak{k}$. Furthermore, $Y_i=y_{\alpha_i}$ for $i=1,\ldots,r$.
The following lemma is now immediate.
%%%%%%%%%%%%%%%%%%%%%
\begin{lem}
\[
\mathfrak{k}_{\mathbb{Z}}:=\{x\in\mathfrak{g}_{\mathbb{Z}}\,\, | \,\, \omega(x)=x\}
\]
is an integral form of $\mathfrak{k}$ with $\mathbb{Z}$-basis $\{y_\alpha\}_{\alpha\in\Phi_+}$.
The structure constants of $\mathfrak{k}_{\mathbb{Z}}$ with respect to $\{y_\alpha\}_{\alpha\in\Phi_+}$
are given by
\begin{equation}\label{yalphastructure}
[y_\alpha,y_\beta]=\kappa_{\alpha,\beta}y_{\alpha+\beta}-\kappa_{\alpha,-\beta}y_{\alpha-\beta}\qquad (\alpha,\beta\in\Phi_+: \alpha\not=\beta)
\end{equation}
(with the natural interpretations of the right hand side when $\alpha+\beta\not\in\Phi$ and/or $\alpha-\beta\not\in\Phi$).
\end{lem}
%%%%%%%%%%%%%%%%%%%%%%

Chevalley involutions of $\mathfrak{g}$ are complex linear extensions of Cartan involutions of
split real forms of $\mathfrak{g}$.  In our present notations the split real form $\mathfrak{g}_0$
of $\mathfrak{g}$ is defined as the real span of the Chevalley basis $\{h_i,e_\alpha\}_{1\leq i\leq r,\alpha\in\Phi}$.
It contains the real form $\mathfrak{h}_0:=\bigoplus_{i=1}^r\mathbb{R}h_i$ of $\mathfrak{h}$
as its Cartan subalgebra, and $\omega_0:=\omega|_{\mathfrak{g}_0}$ is the Cartan involution of $\mathfrak{g}_0$ containing $\mathfrak{h}_0$ in its $-1$-eigenspace. The fix-point Lie subalgebra $\mathfrak{k}_0$ of $\mathfrak{g}_0$ with respect to
$\omega_0$ is generated as real Lie algebra by $Y_i:=e_i-f_i$
($i=1,\ldots,r$), and its complexification is isomorphic to $\mathfrak{k}$.
The real version of Theorem \ref{mainTHM} holds true, with $\mathfrak{k}$ replaced by $\mathfrak{k}_0$ and $\mathcal{L}(A)$ (Definition \ref{DefgenOns}) defined over the real numbers.

The fix-point Lie subalgebras $\mathfrak{k}$
can now be described using the explicit description of $\mathfrak{k}_0$,
see \cite[Appendix C]{Kn} (note that in the list of properties of
$\mathfrak{g}_0\simeq \mathfrak{so}(2p+1,2q+1)$ in \cite[Appendix C, page 528]{Kn}, 
the fact that $\mathfrak{so}(2p+1,2p+1)$ is a split real form is missing). It leads to the 
following table (the type refers to the type of the Cartan matrix $A$, i.e. the type of the simple Lie algebra $\mathfrak{g}$).\\

\begin{enumerate}
\item[] {\bf Type} $A_r$ ($r\geq 2$):  $\mathfrak{k}\simeq\mathfrak{so}_{r+1}(\mathbb{C})$.
\item[] {\bf Type} $B_r$ ($r\geq 3$): $\mathfrak{k}\simeq \mathfrak{so}_r(\mathbb{C})\oplus
\mathfrak{so}_{r+1}(\mathbb{C})$.
\item[] {\bf Type} $C_r$ ($r\geq 1$): $\mathfrak{k}\simeq\mathfrak{gl}_r(\mathbb{C})$.
\item[] {\bf Type} $D_r$ ($r\geq 4$):  $\mathfrak{k}\simeq \mathfrak{so}_r(\mathbb{C})\oplus
\mathfrak{so}_r(\mathbb{C})$.
\item[] {\bf Type} $E_6$: $\mathfrak{k}\simeq \mathfrak{sp}_4(\mathbb{C})$.
\item[] {\bf Type} $E_7$: $\mathfrak{k}\simeq \mathfrak{sl}_8(\mathbb{C})$.
\item[] {\bf Type} $E_8$: $\mathfrak{k}\simeq \mathfrak{so}_{16}(\mathbb{C})$.
\item[] {\bf Type} $F_4$: $\mathfrak{k}\simeq \mathfrak{sp}_3(\mathbb{C})\oplus \mathfrak{sl}_2(\mathbb{C})$.
\item[] {\bf Type} $G_2$; $\mathfrak{k}\simeq \mathfrak{sl}_2(\mathbb{C})\oplus\mathfrak{sl}_2(\mathbb{C})$.
\end{enumerate}
\vspace{.4cm}

{}From the list it is clear that $\mathfrak{k}$ is reductive, and that $\mathfrak{k}$ is semisimple unless the Cartan matrix $A$
is of type $C_r$ ($r\geq 1$).

%%%%%%%%%%%%%%%%%
\begin{rema}
Dolan-Grady type presentations of the quantum analogue of $\mathfrak{k}$ for any symmetric pair
$(\mathfrak{g},\mathfrak{k})$ were considered in \cite{Le,K2}. They play an important role in quantum harmonic analysis because of 
the intrinsic rigidity of quantum symmetric pairs. The quantized universal enveloping algebra $U_q(\mathfrak{g})$, defined by quantizing a Serre presentation of $U(\mathfrak{g})$, depends on a distinguished choice of 
Cartan subalgebra $\mathfrak{h}$. In addition, the quantum analogue of a symmetric pair requires fixing a representative of the isomorphism class of the involution that stabilizes $\mathfrak{h}$
and that has the additional property that its $-1$-eigenspace in $\mathfrak{h}$ is of maximal 
dimension
(the special case under consideration in this paper
corresponds to the most extreme case that the whole Cartan subalgebra $\mathfrak{h}$ is contained in the $-1$-eigenspace of the involution).
It is in this setup that the Serre type
presentation of $\mathfrak{g}$ (resp. $U_q(\mathfrak{g})$) does not induce a Serre type presentation of (the quantum analogue of) $\mathfrak{k}$, but instead leads to Dolan-Grady type presentations.
Note the recent paper \cite{Le2} in which an important first step is made to reveal the reductive nature of $\mathfrak{k}$ in the quantum context. 
\end{rema}
%%%%%%%%%%%%%%%%%%%%%%%%%%

So on the one hand $\mathfrak{k}$ is reductive, hence admits a Serre type presentation, while on the other hand $\mathfrak{k}$ admits an Dolan-Grady type presentation (Theorem \ref{mainTHM}\,{\bf a}). The interplay between these two presentations of $\mathfrak{k}$ is worked out for type $C_r$ ($r\geq 1$)
in Section \ref{section4}.

%%%%%%%%%%%%%%%%%%%%%%%
\subsection{The untwisted affine case}\label{affineSection}
%%%%%%%%%%%%%%%%%%%%%%%%

Let $A$ be an indecomposable generalized Cartan matrix of untwisted affine type $X_r^{(1)}$ (see \cite[\S 4.8, Table Aff 1]{Ka}). 
By the affine version \cite[Thm. 7.4]{Ka} of Serre's Theorem, the corresponding Kac-Moody Lie algebra $\mathfrak{g}(A)$ is an untwisted affine Lie algebra $\widetilde{\mathfrak{g}}$.
We first describe the associated fix-point Lie subalgebra $\widetilde{\mathfrak{k}}:=
\mathfrak{k}(A)$
in terms of
the loop presentation of $\widetilde{\mathfrak{g}}$. 

We recall the loop presentation of the untwisted affine Lie algebras following \cite[Chpt. 7]{Ka}.
The starting point is a simple Lie algebra $\mathfrak{g}$ of type $X_r$ (we will freely use the notations
from the previous subsection regarding the structure theory of $\mathfrak{g}$).
The loop algebra of $\mathfrak{g}$ is 
\[
L\mathfrak{g}=\mathfrak{g}\otimes\mathbb{C}[t,t^{-1}].
\]
For $x\in\mathfrak{g}$ and $k\in\mathbb{Z}$ we write $x[k]:=x\otimes t^k$. Let 
\[
\widehat{\mathfrak{g}}:=L\mathfrak{g}\oplus\mathbb{C}c
\]
be its unique nontrivial central extension. Its Lie bracket is determined by
\begin{equation}\label{Lieaffine}
[x[k],y[m]]=[x,y][k+m]+k\delta_{k,-m}(x,y)c.
\end{equation}
The affine Lie algebra $\widetilde{\mathfrak{g}}$ is the extension of $\widehat{\mathfrak{g}}$ by the derivation $t\frac{d}{dt}$, i.e.
\[
\widetilde{\mathfrak{g}}:=\widehat{\mathfrak{g}}\oplus\mathbb{C}d
\]
with $[d,c]=0$ and $[d,x[m]]=mx[m]$. We extend $\bigl(\cdot,\cdot\bigr): \mathfrak{g}\times\mathfrak{g}\rightarrow\mathbb{C}$ to 
a nondegenerate invariant symmetric form on $\widetilde{\mathfrak{g}}\times\widetilde{\mathfrak{g}}$ by 
\[
(x[k],y[m])=\delta_{k,-m}(x,y),\quad (c,d)=1,\quad (c,c)=0=(d,d),\quad (c,x[k])=0=(d,x[k]).
\]
The abelian Lie subalgebras $\widehat{\mathfrak{h}}=\mathfrak{h}\oplus\mathbb{C}c$ of $\widehat{\mathfrak{g}}$ and
$\widetilde{\mathfrak{h}}=\widehat{\mathfrak{h}}\oplus\mathbb{C}d$ of $\widetilde{\mathfrak{g}}$
play the role of Cartan subalgebras. Define $\Lambda_0,\delta\in\widetilde{\mathfrak{h}}^*$ by
\begin{equation*}
\begin{split}
\Lambda_0(c)&=1,\qquad \Lambda_0(\mathfrak{h})=0=\Lambda_0(d),\\
\delta(d)&=1,\qquad \delta(\mathfrak{h})=0=\delta(c).
\end{split}
\end{equation*}
We identify $\mathfrak{h}^*$ with the subspace of $\widetilde{\mathfrak{h}}^*$ vanishing at $\mathbb{C}c\oplus\mathbb{C}d$.
Then 
\[
\widetilde{\mathfrak{h}}^*=\mathfrak{h}^*\oplus\mathbb{C}\Lambda_0\oplus\mathbb{C}\delta.
\]

Since $(\cdot,\cdot)\vert_{\widetilde{\mathfrak{h}}\times\widetilde{\mathfrak{h}}}$ is nondegenerate there exists for $\lambda\in\widetilde{\mathfrak{h}}^*$
a unique $t_\lambda\in\widetilde{\mathfrak{h}}$ such that $\lambda(h)=(h,t_\lambda)$ for all $h\in\widetilde{\mathfrak{h}}$. For $\lambda\in\mathfrak{h}^*$ we have 
$t_\lambda\in\mathfrak{h}$, which coincides with the element $t_\lambda$ as defined in the previous subsection.
Note furthermore that
\[
t_{\Lambda_0}=d,\qquad t_\delta=c.
\]
The bilinear form on $\widetilde{\mathfrak{h}}^*$ obtained from the nondegenerate symmetric bilinear form
$(\cdot,\cdot)\vert_{\widetilde{\mathfrak{h}}\times\widetilde{\mathfrak{h}}}$
by dualizing is again denoted by $(\cdot,\cdot)$. It satisfies
\[
(\lambda,\mu):=\lambda(t_\mu)=(t_\lambda,t_\mu),\qquad \lambda,\mu\in\widetilde{\mathfrak{h}}^*.
\]
Note that
\[
(\Lambda_0,\Lambda_0)=0=(\delta,\delta),\qquad (\Lambda_0,\delta)=1,\qquad (\Lambda_0,\mathfrak{h}^*)=0=(\delta,\mathfrak{h}^*).
\]

The affine root system $\widetilde{\Phi}=\widetilde{\Phi}^{\textup{re}}\cup\widetilde{\Phi}^{\textup{im}}\subset\widetilde{\mathfrak{h}}^*$ of $\widetilde{\mathfrak{g}}$ is 
%defined by
\[
\widetilde{\Phi}^{\textup{re}}:=\{\alpha+k\delta\,\, | \,\, \alpha\in\Phi, k\in\mathbb{Z}\},\qquad
\widetilde{\Phi}^{\textup{im}}:=\{k\delta\,\, | \,\, k\in\mathbb{Z}\setminus\{0\}\}.
\]
We take as associated set of positive affine roots $\widetilde{\Phi}_+:=\widetilde{\Phi}^{\textup{re}}_+\cup\widetilde{\Phi}^{\textup{im}}_+$ with 
\[
\widetilde{\Phi}^{\textup{re}}_+:=\Phi_+\cup\{\alpha+k\delta\,\, | \,\, \alpha\in\Phi, k\in\mathbb{Z}_{>0}\},\qquad
\widetilde{\Phi}^{\textup{im}}_+:=\mathbb{Z}_{>0}\delta.
\]
We denote the negative affine roots by $\widetilde{\Phi}_-=\widetilde{\Phi}^{\textup{re}}_-\cup\widetilde{\Phi}^{\textup{im}}_-$. The corresponding simple roots are
$\widetilde{\Pi}:=\{\alpha_0,\alpha_1,\ldots,\alpha_r\}$ with the additional affine simple root given by
\[
\alpha_0:=-\theta+\delta
\]
with $\theta\in\Phi_+$ the highest root of $\Phi$. Note that $\theta$ is a long root, hence $(\theta,\theta)=2$ by our convention on the normalisation of $(\cdot,\cdot)$.

The root space decomposition of $\widetilde{\mathfrak{g}}$ is 
\[
\widetilde{\mathfrak{g}}=\widetilde{\mathfrak{h}}\oplus\bigoplus_{\gamma\in\widetilde{\Phi}}\widetilde{\mathfrak{g}}_{\gamma}
\]
with root spaces $\widetilde{\mathfrak{g}}_\gamma:=\{x\in\widetilde{\mathfrak{g}}\,\, | \,\, [h,x]=\gamma(h)x\,\,\, \forall\, h\in\widetilde{\mathfrak{h}} \}$ for $\gamma\in\widetilde{\Phi}$.
The root spaces are concretely given by
\begin{equation*}
\begin{split}
\widetilde{\mathfrak{g}}_{\alpha+k\delta}&=\mathfrak{g}_\alpha\otimes t^k\qquad (\alpha\in\Phi,\,\, k\in\mathbb{Z}),\\
\widetilde{\mathfrak{g}}_{k\delta}&=\mathfrak{h}\otimes t^k\qquad\,\,\, (k\in\mathbb{Z}\setminus\{0\}).
\end{split}
\end{equation*}
In particular, $m(\gamma)=1$ if $\gamma\in\widetilde{\Phi}^{\textup{re}}$ and $m(\gamma)=r$ if 
$\gamma\in\widetilde{\Phi}^{\textup{im}}$.

The description of $\widetilde{\mathfrak{g}}$ as a Kac-Moody algebra is obtained as follows. Define 
\[
h_\gamma:=\frac{2t_\gamma}{(\gamma,\gamma)},\qquad \gamma\in\widetilde{\Phi}^{\textup{re}}
\]
and write $\widetilde{\Pi}^\vee:=\{h_0,h_1,\ldots,h_r\}$ with 
$h_j:=h_{\alpha_j}$ for $j=0,\ldots,r$.  Note that for  $\gamma=\alpha\in\Phi$ and $j=i\in\{1,\ldots,r\}$ the elements $h_\gamma$ and $h_j$ are the elements $h_\alpha\in\mathfrak{h}$ and $h_i\in\mathfrak{h}$ as defined in the previous subsection.
In addition,
\[
h_0=t_{\alpha_0}=c-t_\theta=c-h_\theta
\]
since $(\alpha_0,\alpha_0)=(-\theta+\delta,-\theta+\delta)=(\theta,\theta)=2$. 

Then $(\widetilde{\mathfrak{h}},\widetilde{\Pi},
\widetilde{\Pi}^\vee)$ is a realization of the affine Cartan matrix $\widetilde{A}:=
\bigl(\alpha_j(h_i)\bigr)_{i,j=0}^r$ of type $X_r^{(1)}$ (it contains the Cartan matrix
$A=\bigl(\alpha_j(h_i)\bigr)_{i,j=1}^r$ of type $X_r$).
The affine Cartan matrix $\widetilde{A}$ is symmetrizable by \cite[Lem. 4.6]{Ka}.
Furthermore, the rank of $\widetilde{A}$ is $r$ and the dimension of $\widetilde{\mathfrak{h}}$
is $r+2=2n-r$ with $n:=r+1$.

Let $e_i\in\mathfrak{g}_{\alpha_i}$ and
$f_i\in\mathfrak{g}_{-\alpha_i}$ ($i=1,\ldots,r$) be the Chevalley generators of $\mathfrak{g}$ 
and $\omega$ the corresponding Chevalley involution of $\mathfrak{g}$.
Choose 
\[
E_0\in\mathfrak{g}_{-\theta}
\]
such that $(E_0,\omega(E_0))=-1$ and set
\[
F_0:=-\omega(E_0)\in\mathfrak{g}_\theta.
\]
Then $(E_0,F_0)=1$ and $[E_0,F_0]=-h_{\theta}$. Define now 
$e_0\in\widetilde{\mathfrak{g}}_{\alpha_0}$ and $f_0\in\widetilde{\mathfrak{g}}_{-\alpha_0}$
by 
\[
e_0:=E_0[1],\qquad f_0:=F_0[-1].
\]
Note that $[e_0,f_0]=h_0$. For $i=1,\ldots,r$ interpret $e_i$ and $f_i$ as elements in 
$\widetilde{\mathfrak{g}}$ by the canonical Lie algebra embedding $\mathfrak{g}\hookrightarrow\widetilde{\mathfrak{g}}$,
$x\mapsto x\otimes 1$.
Then \cite[Thm. 7.4]{Ka} shows that $\widetilde{\mathfrak{g}}
\simeq\mathfrak{g}(\widetilde{A})$ by identifying 
$e_0,\ldots,e_r, f_0,\ldots,f_r$ with the Chevalley generators of $\widetilde{\mathfrak{g}}$.
We will freely use the identification $\widetilde{\mathfrak{g}}\simeq\mathfrak{g}(\widetilde{A})$ in this subsection. 

We write $\widetilde{\omega}$ for the Chevalley involution of $\widetilde{\mathfrak{g}}$. It is characterized by
$\widetilde{\omega}(e_j)=-f_j$ and $\widetilde{\omega}(h_j)=-h_j$ for $j=0,\ldots,r$. Write
 $\widetilde{\mathfrak{k}}$ for the fix-point Lie subalgebra. The corresponding generators $Y_j$ ($0\leq j\leq r$) of $\widetilde{\mathfrak{k}}$ are $Y_j=e_j-f_j$ for $j=0,\ldots,r$. Theorem \ref{mainTHM}{\bf a}
leads to the Dolan-Grady type presentation of $\widetilde{\mathfrak{k}}$ in terms of the generators $Y_j$ ($j=0,\ldots,r$) by the identification of $\widetilde{\mathfrak{k}}$ with the generalized Onsager algebra $\widetilde{\mathcal{L}}:=\mathcal{L}(\widetilde{A})$.
 
 Note that $\widetilde{\omega}$ extends the Chevalley involution $\omega$ of
$\mathfrak{g}$, and that
\[
\widetilde{\omega}(x[k])=\omega(x)[-k], \quad \widetilde{\omega}(c)=-c,\quad \widehat{\omega}(d)=-d
\]
for $x\in\mathfrak{g}$ and $k\in\mathbb{Z}$ (see \cite[\S 7.6]{Ka}). 

Let $\{h_i,e_\alpha\}_{1\leq i\leq r, \alpha\in\Phi}$ be a Chevalley basis of $\mathfrak{g}$ 
satisfying
\begin{enumerate}
\item[{\bf a.}] $e_{\alpha_j}=e_j$ and $e_{-\alpha_j}=f_j$ for $j=1,\ldots,r$,
\item[{\bf b.}] $e_\theta=F_0$,
\item[{\bf c.}] $\omega(e_\alpha)=-e_{-\alpha}$ for all $\alpha\in\Phi$.
\end{enumerate}
Such a Chevalley basis exists by \cite[\S 25.2]{Hu}
since $F_0\in\mathfrak{g}_\theta$ and $(F_0,\omega(F_0))=-1$.
Then 
\[
\widetilde{\mathfrak{g}}_{\mathbb{Z}}:=
\mathbb{Z}c\oplus\mathbb{Z}d\oplus\bigoplus_{1\leq i\leq r, k\in\mathbb{Z}}
\mathbb{Z}h_i[k]\oplus\bigoplus_{\alpha\in\Phi, k\in\mathbb{Z}}\mathbb{Z}e_\alpha[k]
\]
is a $\widetilde{\omega}$-stable integral form of $\widetilde{\mathfrak{g}}$, see \cite{G}.
Furthermore,
\begin{equation*}
\widetilde{\mathfrak{k}}_{\mathbb{Z}}:=
\{x\in\widetilde{\mathfrak{g}}_{\mathbb{Z}} \,\,\, | \,\,\, \widetilde{\omega}(x)=x\}
=\bigoplus_{\gamma\in\widetilde{\Phi}_+, 1\leq i\leq m(\gamma)}\mathbb{Z}y_\gamma^{(i)}
\end{equation*}
is an integral form of $\widetilde{\mathfrak{k}}$, with the elements $y_\gamma^{(i)}$ defined
by
\begin{equation}\label{integralbasisk}
y_\gamma^{(i)}:=
\begin{cases}
e_\alpha[k]-e_{-\alpha}[-k]\quad &\hbox{ if }\,\,
 \gamma=\alpha+k\delta\in\widetilde{\Phi}^{\textup{re}}_+,\\
h_i[k]-h_i[-k]\quad &\hbox{ if }\,\, \gamma=k\delta\in\widetilde{\Phi}^{\textup{im}}_+
\end{cases}
\end{equation}
for $i=1,\ldots,m(\gamma)$ (note that $m(\gamma)=1$ for 
$\gamma\in\widetilde{\Phi}_+^{\textup{re}}$). We write
$y_\gamma:=y_\gamma^{(1)}$ ($\gamma\in\widetilde{\Phi}_+^{\textup{re}}$), 
$y_0^{(i)}:=0$ ($1\leq i\leq r$) and $y_{-\gamma}^{(i)}:=-y_\gamma^{(i)}$
for $\gamma\in\widetilde{\Phi}_+$ and $1\leq i\leq m(\gamma)$. Note that
$Y_j=y_{\alpha_j}$ for $0\leq j\leq r$.

For $\alpha\in\Phi$ let $k_i(\alpha)\in\mathbb{Z}$ such that 
$h_\alpha=\sum_{i=1}^rk_i(\alpha)h_i$.
Recall the constants $\kappa_{\alpha,\beta}\in\mathbb{Z}$
such that 
\[
[e_\alpha,e_\beta]=\kappa_{\alpha,\beta}e_{\alpha+\beta},\qquad (\alpha,\beta\in\Phi: \alpha+\beta\not=0),
\]
cf. \eqref{calphabeta} (recall that by convention $\kappa_{\alpha,\beta}$ is zero if $\alpha+\beta\not\in\Phi$).

%%%%%%%%%%%%%%%%%%%%%%%%%%%%%%%%%%%%%%%%%%%
\begin{prop}
The structure constants of $\widetilde{\mathfrak{k}}_{\mathbb{Z}}$ with respect to its $\mathbb{Z}$-basis $\{y_\gamma^{(i)}\}_{\gamma\in\widetilde{\Phi}_+, 1\leq i\leq m(\gamma)}$ are determined by: 
\begin{equation}\label{rel1}
\begin{split}
\lbrack y_{\alpha+\ell\delta},y_{\alpha+m\delta}\rbrack&=
\sum_{i=1}^rk_i(\alpha)y_{(m-\ell)\delta}^{(i)},\\
\lbrack y_{\alpha+\ell\delta},y_{-\alpha+m\delta}\rbrack&=
\sum_{i=1}^rk_i(\alpha)y_{(m+\ell)\delta}^{(i)}
\end{split}
\end{equation}
for $\alpha\in\Phi$ and $\ell,m\in\mathbb{Z}$, 
\begin{equation}\label{rel2}
\lbrack y_{\alpha+\ell\delta},y_{\beta+m\delta}\rbrack=\kappa_{\alpha,\beta}y_{\alpha+\beta+(\ell+m)\delta}
-\kappa_{\alpha,-\beta}y_{\alpha-\beta+(\ell-m)\delta}
\end{equation}
for $\alpha,\beta\in\Phi$ with $\alpha\not=\pm\beta$ and $\ell,m\in\mathbb{Z}$,
\begin{equation}\label{rel3}
\lbrack y_{\ell\delta}^{(i)},y_{\alpha+m\delta}\rbrack=
\alpha(h_i)y_{\alpha+(\ell+m)\delta}-\alpha(h_i)y_{\alpha+(m-\ell)\delta}
\end{equation}
for $\alpha\in\Phi$, $\ell,m\in\mathbb{Z}$ and $1\leq i\leq r$, and
\begin{equation}\label{rel4}
[y_{\ell\delta}^{(i)}, y_{m\delta}^{(j)}]=0
\end{equation}
for $\ell,m\in\mathbb{Z}$ and $1\leq i,j\leq r$.
\end{prop}
%%%%%%%%%%%%%%%%%%%%%%%%%%%%%%
\begin{proof}
This follows from direct computations using \eqref{Lieaffine} and the relations
$[h_i,e_\alpha]=\alpha(h_i)e_\alpha$,  $[e_\alpha,e_{-\alpha}]=h_\alpha$
and \eqref{calphabeta}.
\end{proof}
%%%%%%%%%%%%%%%%%%%%%%%%%%%%%%%
\begin{eg}
Consider the special case $\widetilde{\mathfrak{g}}=\widetilde{\mathfrak{sl}}_2$ with the
affine Cartan matrix $\widetilde{A}=\left(\begin{matrix} 2 & -2\\ -2 & 2\end{matrix}\right)$ of type
$A_1^{(1)}$. In this case $\widetilde{\mathfrak{k}}$ is the Onsager algebra and $\widetilde{\mathfrak{k}}\simeq\widetilde{\mathcal{L}}$ gives it Dolan-Grady type presentation.
Write for $m\in\mathbb{Z}$,
\[
A_m:=y_{\alpha_1+m\delta},\qquad G_m:=y_{m\delta}^{(1)}.
\]
Note that $A_m=-y_{\alpha_0-(m+1)\delta}$, $G_{-m}=-G_m$ and
$G_0=0$. In particular, 
\[
\{A_\ell\}_{\ell\in\mathbb{Z}}\cup\{G_m\}_{m\in\mathbb{Z}_{>0}}
\] 
is a 
$\mathbb{Z}$-basis of $\widetilde{\mathfrak{k}}_{\mathbb{Z}}$. By the previous proposition,
\[
[A_\ell,A_m]=G_{m-\ell},\quad [G_\ell,A_m]=2A_{\ell+m}-2A_{m-\ell},\quad
[G_\ell,G_m]=0
\]
for $\ell,m\in\mathbb{Z}$, which gives the Onsager \cite{O} presentation of $\widetilde{\mathfrak{k}}$.
\end{eg}
%%%%%%%%%%%%%%%%%%%%%%%%%%%%%%%

%%%%%%%%%%%%%%%%%%%%%%%%%%%%%%%%%%%%%%%%%%%%%%
\section{The symplectic case}\label{section4}
%%%%%%%%%%%%%%%%%%%%%%%%%%%%%%%%%%%%%%%%%%%%%%%
We start this section by explicitly comparing the Serre presentation and the Dolan-Grady type presentation of $\mathfrak{k}$ when $A$ is of type $C_r$ ($r\geq 1$). We will explicitly evaluate its one-dimensional representations on the Chevalley basis of $\mathfrak{k}$. 
We end the section by considering the untwisted affine case. 

Let $\mathfrak{g}=\mathfrak{sp}_r(\mathbb{C})$ be the symplectic Lie algebra consisting of $2r\times 2r$ complex-valued matrices
\begin{equation*}
\left(\begin{matrix} B & C\\ D & -B^T\end{matrix}\right)
\end{equation*}
with $B,C,D\in\mathfrak{gl}_r(\mathbb{C})$ and $C^T=C$, $D^T=D$. Let 
$\mathfrak{h}$ be the Cartan subalgebra of $\mathfrak{g}$ consisting of 
the diagonal matrices in $\mathfrak{sp}_r(\mathbb{C})$. 

Define $\epsilon_j\in\mathfrak{h}^*$ by
\[
\epsilon_j\left(\begin{matrix}E_{k,k} & 0\\ 0 & -E_{k,k}\end{matrix}\right):=\delta_{j,k}
\]
with $E_{k,\ell}$ ($1\leq k,\ell\leq r$) the matrix units in $\mathfrak{gl}_r(\mathbb{C})$.
The root system $\Phi$ of $(\mathfrak{g},\mathfrak{h})$ is 
\[
\Phi=\{\pm(\epsilon_k\pm\epsilon_\ell)\}_{1\leq k\not=\ell\leq r}\cup\{\pm 2\epsilon_j\}_{j=1}^r
\]
($\Phi=\{\pm 2\epsilon_1\}$ for $r=1$). Note that the bilinear form on $\mathfrak{h}^*$ induced from
the normalized invariant symmetric bilinear form on $\mathfrak{sp}_r(\mathbb{C})$ is given by $\bigl(\epsilon_i,\epsilon_j\bigr)=\frac{1}{2}\delta_{i,j}$.

Take as simple roots $\Pi=\{\alpha_k\}_{k=1}^r\subset \mathfrak{h}^\ast$ with
\[
\alpha_k:=\epsilon_k-\epsilon_{k+1}\quad (1\leq k<r),\qquad \alpha_r:=2\epsilon_r.
\]
Then $\Pi^\vee=\{h_j=h_{\alpha_j}\}_{j=1}^r\subset\mathfrak{h}$ is explicitly given by
\[
h_j=\left(\begin{matrix} E_{j,j}-E_{j+1,j+1} & 0\\ 0 & -E_{j,j}+E_{j+1,j+1}\end{matrix}\right),
\qquad\,\, h_r=\left(\begin{matrix} E_{r,r} & 0\\ 0 & -E_{r,r}\end{matrix}\right)
\]
for $1\leq j<r$. 
As Chevalley generators $e_1,\ldots, e_r,f_1,\ldots,f_r$ of $\mathfrak{sp}_r(\mathbb{C})$ we take
\begin{equation*}
\begin{split}
e_j&:=\left(\begin{matrix} E_{j,j+1} & 0\\ 0 & -E_{j+1,j}\end{matrix}\right),\qquad\qquad
\qquad
e_r:=\left(\begin{matrix} 0 & E_{r,r}\\ 0 & 0\end{matrix}\right),\\
f_j&:=\left(\begin{matrix} E_{j+1,j} & 0\\ 0 & -E_{j,j+1}\end{matrix}\right),\qquad\qquad
\qquad
f_r:=\left(\begin{matrix} 0 & 0\\ E_{r,r} & 0\end{matrix}\right)
\end{split}
\end{equation*}
for $1\leq j<r$. The associated Chevalley involution $\omega$ then becomes $\omega(X):=-X^T$ ($X\in\mathfrak{sp}_r(\mathbb{C})$). An extension of $\{h_j,e_j,f_j\}_{j=1}^r$ to a Chevalley basis $\{h_j,e_\alpha\}_{1\leq j\leq r, \alpha\in\Phi}$ of $\mathfrak{sp}_r(\mathbb{C})$
such that $e_{\alpha_j}=e_j$, $e_{-\alpha_j}=f_j$ ($1\leq j\leq r$) and $\omega(e_\alpha)=-e_{-\alpha}$ for all $\alpha\in\Phi$, is given by
\begin{equation*}
\begin{split}
e_{\epsilon_k-\epsilon_\ell}&:=\left(\begin{matrix} E_{k,\ell} & 0\\ 0 & -E_{\ell,k}\end{matrix}\right),\qquad
e_{\epsilon_k+\epsilon_\ell}:=\left(\begin{matrix} 0 & E_{k,\ell}+E_{\ell,k}\\ 0 & 0\end{matrix}\right),
\qquad
e_{-\epsilon_k-\epsilon_\ell}:=\left(\begin{matrix} 0 & 0\\ E_{k,\ell}+E_{\ell,k} & 0\end{matrix}\right)\\
e_{2\epsilon_j}&:=\left(\begin{matrix} 0 & E_{j,j}\\ 0 & 0\end{matrix}\right),
\qquad\qquad e_{-2\epsilon_j}:=\left(\begin{matrix} 0 & 0\\  E_{j,j} & 0\end{matrix}\right)
\end{split}
\end{equation*}
for $1\leq k\not=\ell\leq r$ and $1\leq j\leq r$.

The fix-point Lie subalgebra is
\begin{equation}\label{fixedpointsp}
\mathfrak{k}=
\left\{\left(\begin{matrix} B & C\\ -C & B\end{matrix}\right) \,\, \vline \,\, B^T=-B \,\,\,\, \&\,\,\,\, C^T=C\right\},
\end{equation}
with associated generators $Y_1,\ldots,Y_r$ given by $Y_j:=e_j-f_j$ ($1\leq j\leq r$) and 
Chevalley-type basis $\{y_\alpha\}_{\alpha\in\Phi_+}$ given by $y_\alpha:=e_\alpha-e_{-\alpha}$ for 
$\alpha\in\Phi$. Note that $Y_j=e_{\alpha_j}$ for $1\leq j\leq r$.

The map
\begin{equation*}
\eta\left(\begin{matrix} B & C\\ -C & B\end{matrix}\right):=B+iC
\end{equation*}
for $B,C\in\mathfrak{gl}_r(\mathbb{C})$ satisfying $B^T=-B$ and $C^T=C$ defines a
Lie algebra isomorphism $\eta: \mathfrak{k}\overset{\sim}{\longrightarrow} \mathfrak{gl}_r(\mathbb{C})$. 
Note that $\eta(h_j)=E_{j,j}-E_{j+1,j+1}$ ($1\leq i<r$), $\eta(h_r)=E_{r,r}$ and
\[
\eta(Y_j)=E_{j,j+1}-E_{j+1,j},\qquad \eta(Y_r)=iE_{r,r}
\]
for $1\leq j<r$. Furthermore,
\begin{equation}\label{sigmabasis}
\begin{split}
\eta(y_{\epsilon_j-\epsilon_k})&=\bigl(E_{j,k}-E_{k,j}\bigr),\\
\eta(y_{\epsilon_j+\epsilon_k})&=i\bigl(E_{j,k}+E_{k,j}\bigr),\\
\eta(y_{2\epsilon_{\ell}})&=iE_{\ell,\ell}
\end{split}
\end{equation}
for $1\leq j<k\leq r$ and $1\leq \ell\leq r$.

 Transporting the Dolan-Grady type presentation of $\mathfrak{k}$ (see Theorem
\ref{mainTHM}) through the Lie algebra isomorphism $\eta$ gives
the following 
presentation of the general linear Lie algebra $\mathfrak{gl}_r(\mathbb{C})$.
%%%%%%%%%%%%%%%%%%%%%%%%%%%%%%%%%%%%%%%%%%%%
\begin{cor}\label{strangepres}
Write $Z:=\sum_{\ell=1}^rE_{\ell,\ell}$ for the generator of the centre of $\mathfrak{gl}_r(\mathbb{C})$.
Set $K_j:=E_{j,j+1}-E_{j+1,j}$ for $1\leq j<r$ and write
\[
K_r:=iE_{r,r}=\frac{i}{r}\Bigl(Z-\sum_{j=1}^{r-1}j\eta(h_j)\Bigr).
\]
Then $\mathfrak{gl}_r(\mathbb{C})$ is generated by $K_1,\ldots,K_r$. The defining relations of $\mathfrak{gl}_r(\mathbb{C})$ in terms of the generators $K_1,\ldots,K_r$ are 
\begin{equation*}
\begin{split}
[K_j,K_k]&=0,\qquad\qquad\,\, |j-k|>1,\\
[K_j,[K_j,K_{j+1}]]&=-K_{j+1},\qquad 1\leq j<r-1,\\
[K_{j+1},[K_{j+1},K_{j}]]&=-K_{j},\qquad\quad 1\leq j<r,\\
[K_{r-1},[K_{r-1},[K_{r-1},K_r]]]&=-4[K_{r-1},K_r].
\end{split}
\end{equation*}
\end{cor}
%%%%%%%%%%%%%%%%%%%%%%%%%%%%%%%%%%%%%%%%%%%
\begin{rema}
In a similar manner one can write down an explicit Dolan-Grady type presentation
of $\mathfrak{so}_{r+1}(\mathbb{C})$ using the fact that $\mathfrak{so}_{r+1}(\mathbb{C})$
is the fix-point Lie subalgebra for the Chevalley involution of $\mathfrak{sl}_{r+1}(\mathbb{C})$ (see Example \ref{example}). 
\end{rema}
%%%%%%%%%%%%%%%%%%%%%%%%%%%%%%%%%%%%%%%%%

Write $\Phi=\Phi^s\cup\Phi^\ell$ 
with $\Phi^s$ the set of short roots and $\Phi^\ell$ the set of long roots. For $r=1$ we set
$\Phi^\ell=\Phi$. 
The short and long positive roots are
\[
\Phi^s_+=\{\epsilon_j\pm\epsilon_k\,\, | \,\, 1\leq j<k\leq r\},\qquad \Phi^\ell_+=
\{2\epsilon_\ell \,\, | \,\, 1\leq \ell\leq r\}.
\]

Note that $\mathcal{E}_A=\{r\}$ for the Cartan matrix $A=(\alpha_j(h_i))_{i,j=1}^r$ of 
$\mathfrak{sp}_r(\mathbb{C})$. Hence $\mathfrak{k}$ has a one-dimensional space
$\textup{ch}(\mathfrak{k})$ of one-dimensional representations, see Proposition \ref{onedimrep}. 
Concretely, for $t\in\mathbb{C}$ write $\chi_t\in\textup{ch}(\mathfrak{k})$ for the one-dimensional representation such that $\chi_t(Y_i)=0$ for $1\leq i<r$ and $\chi_t(Y_r)=t$ (see Proposition \ref{onedimrep}). The explicit analysis
above now allows one to compute $\chi_t$ on the Chevalley basis vectors $y_\alpha\in\mathfrak{k}$ ($\alpha\in\Phi_+$).
%%%%%%%%%%%%%%%%%%%%%%%%%%%%%%%%%%%%%%%%%%%
\begin{lem}\label{ExplicitChar}
Let $\mathfrak{g}=\mathfrak{sp}_r(\mathbb{C})$ and $\mathfrak{k}$ be the fixed-point Lie subalgebra of $\mathfrak{sp}_r(\mathbb{C})$ with respect to $\omega$ (see \eqref{fixedpointsp}).  Then
\begin{equation*}
\chi_t(y_\alpha)=
\begin{cases}
0\quad &\hbox{ if }\,\, \alpha\in\Phi^s_+,\\
t\quad &\hbox{ if }\,\, \alpha\in\Phi^\ell_+.
\end{cases}
\end{equation*}
\end{lem}
%%%%%%%%%%%%%%%%%%%%%%%%%%%%%%%%%%%%%%%%%%%
\begin{proof}
Consider the one-dimensional representation $\chi_t\circ\eta^{-1}$ of $\mathfrak{gl}_r(\mathbb{C})$. It vanishes on $\mathfrak{sl}_r(\mathbb{C})$, hence
\[
\chi_t(y_\alpha)=0\qquad (\alpha\in\Phi_+^s)
\]
since $\eta(y_\alpha)\in\mathfrak{sl}_r(\mathbb{C})$
for all $\alpha\in\Phi_+^s$. For $1\leq \ell\leq r$ we have
\begin{equation*}
\chi_t(y_{2\epsilon_\ell})=(\chi_t\circ\eta^{-1})(iE_{\ell,\ell})
=(\chi_t\circ\eta^{-1})(iE_{r,r})=\chi_t(Y_r)=t,
\end{equation*}
where we have used that $\chi_t\circ\eta^{-1}$ vanishes on $\mathfrak{sl}_r(\mathbb{C})$ for the
second equality.
\end{proof}
%%%%%%%%%%%%%%%%%%%%%%%%%%%%%%%%%%%%%%%%%%%%

We now extend this result to the associated untwisted affine case $\widetilde{\mathfrak{g}}=
\widetilde{\mathfrak{sp}}_r(\mathbb{C})$.

The highest root $\theta\in \Phi_+^\ell$ for $\mathfrak{g}=\mathfrak{sp}_r(\mathbb{C})$
is $\theta=2\epsilon_1$. Hence 
\[
\widetilde{\Pi}=\{\alpha_0=-2\epsilon_1+\delta,\alpha_1,\ldots,\alpha_r\}
\]
is the set of simple roots of $\widetilde{\mathfrak{sp}}_r(\mathbb{C})$ and
$\widetilde{\Pi}^\vee=\{h_0=c-h_\theta,h_1,\ldots,h_r\}$. Concretely,
\[
h_0=c-\left(\begin{matrix} E_{1,1} & 0\\ 0 & -E_{1,1}\end{matrix}\right).
\]
Take 
\[
E_0:=e_{-2\epsilon_1}=\left(\begin{matrix} 0 & 0\\ E_{1,1} & 0\end{matrix}\right),\qquad F_0:=e_{2\epsilon_1}=\left(\begin{matrix} 0 & E_{1,1}\\ 0 & 0\end{matrix}\right).
\]
As extension of the Chevalley generators $e_1,\ldots,e_r,f_1,\ldots,f_r$ of $\mathfrak{g}=\mathfrak{sp}_r(\mathbb{C})$ to Chevalley generators of $\widetilde{\mathfrak{sp}}_r(\mathbb{C})$ 
(see Subsection \ref{affineSection}) we thus add the generators
\[
e_0:=e_{-2\epsilon_1}[1],\qquad f_0:=e_{2\epsilon_1}[-1].
\]
With this choice of $e_0$ and $f_0$, the Chevalley basis $\{h_i,e_\alpha\}_{1\leq i\leq r, \alpha\in\Phi}$ 
of $\mathfrak{g}=\mathfrak{sp}_r(\mathbb{C})$ satisfies the required properties {\bf a-c} from
Subsection \ref{affineSection}. In particular, it extends to an integral basis of 
$\widetilde{\mathfrak{sp}}_r(\mathbb{C})$ and gives rise to the integral basis
$\{y_\gamma^{(i)}\}_{\gamma\in\widetilde{\Phi}_+, 
1\leq i\leq m(\gamma)}$ of the fix-point Lie subalgebra 
$\widetilde{\mathfrak{k}}$ of $\widetilde{\mathfrak{sp}}_r(\mathbb{C})$ (see \eqref{integralbasisk}).
We have the formulas $Y_j:=e_j-f_j=y_{\alpha_j}$ ($0\leq j\leq r$) for the generators of $\widetilde{\mathfrak{k}}$. Recall the notation $y_{-\gamma}^{(i)}:=-y_\gamma^{(i)}$ for $\gamma\in\widetilde{\Phi}_+$.

Note that $\mathcal{E}_{\widetilde{A}}=\{0,r\}$ for the affine Cartan matrix $\widetilde{A}=(\alpha_j(h_i))_{i,j=0}^r$ of 
$\widetilde{\mathfrak{sp}}_r(\mathbb{C})$. Hence $\widetilde{\mathfrak{k}}$ has a two-dimensional space
$\textup{ch}(\widetilde{\mathfrak{k}})$ of one-dimensional representations, see Proposition \ref{onedimrep}. Concretely, for $s,t\in\mathbb{C}$ write $\chi_{s,t}\in\textup{ch}(\widetilde{\mathfrak{k}})$ for the one-dimensional
representation such that $\chi_{s,t}(Y_0)=s$, $\chi_{s,t}(Y_i)=0$ ($1\leq i<r$) and
$\chi_{s,t}(Y_r)=t$ (see Proposition \ref{onedimrep}). We have the following extension of Lemma \ref{ExplicitChar}.
%%%%%%%%%%%%%%%%%%
\begin{prop}\label{ExplicitCharAff}
Let $\widetilde{\mathfrak{k}}$ be the fix-point Lie subalgebra of 
$\widetilde{\mathfrak{sp}}_r(\mathbb{C})$ with respect to the Chevalley involution $
\widetilde{\omega}$. Then
\begin{equation*}
\begin{split}
\chi_{s,t}\bigl(y_{\pm\alpha+2k\delta}\bigr)&=\pm t \qquad (\alpha\in\Phi_+^\ell,\,\, k\in\mathbb{Z}),\\
\chi_{s,t}\bigl(y_{\mp\alpha+(2k+1)\delta}\bigr)&=\pm s \qquad (\alpha\in\Phi_+^\ell,\,\, 
k\in\mathbb{Z}),\\
\chi_{s,t}\bigl(y_{\alpha+k\delta}\bigr)&=0 \qquad\,\,\,\, (\alpha\in\Phi^s,\,\, k\in\mathbb{Z}).
\end{split}
\end{equation*}
Furthermore, $\chi_{s,t}(y_{k\delta}^{(i)})=0$ for $k\in\mathbb{Z}\setminus\{0\}$ and $1\leq i\leq r$.
\end{prop}
%%%%%%%%%%%%%%%%%%%
\begin{proof}
If $\alpha\in\Phi_+^s$ is not a simple root then it is easy to check that there exists $\beta,\gamma\in
\Phi_+$ such that $\alpha=\beta+\gamma$ and $\beta-\gamma\not\in\Phi$. By \eqref{rel2} we get
$[y_\beta, y_{\gamma+k\delta}]=\kappa_{\beta,\gamma}y_{\alpha+k\delta}$ with 
$\kappa_{\beta,\gamma}\not=0$. Hence $\chi_{s,t}(y_{\alpha+k\delta})=0$ for all $k\in\mathbb{Z}$. 
If $\alpha=\alpha_i\in\Phi_+^s$ ($1\leq i<r$) then there exists $\beta,\gamma\in\Phi_+$ such that
$\alpha_i=\beta-\gamma$ 
and $\beta+\gamma\not\in\Phi$.
Again by \eqref{rel2} we get $\chi_{s,t}(y_{\alpha_i+k\delta})=0$ for all $k\in\mathbb{Z}$. Consequently
$\chi_{s,t}(y_{\alpha+k\delta})=0$ for $\alpha\in\Phi^s$ and $k\in\mathbb{Z}$.

By \eqref{rel3} we have for $\alpha\in\Phi$,
\[
\sum_{i=1}^rk_i(\alpha)\lbrack y_\delta^{(i)}, y_{\alpha+k\delta}\rbrack=
2(y_{\alpha+(k+1)\delta}-y_{\alpha+(k-1)\delta}).
\]
Hence 
\begin{equation}\label{shiftinvariant}
\chi_{s,t}(y_{\alpha+(k+1)\delta})=\chi_{s,t}(y_{\alpha+(k-1)\delta})
\end{equation}
for 
$\alpha\in\Phi$ and $k\in\mathbb{Z}$. In particular, the formula $\chi_{s,t}(y_{\pm\alpha+2k\delta})=\pm t$
for $\alpha\in\Phi_+^\ell$ and $k\in\mathbb{Z}$ is valid if it holds true for $k=0$. This in turn follows from Lemma \ref{ExplicitChar}. 

Next we prove that $\chi_{s,t}(y_{\mp\alpha+(2k+1)\delta})=\pm s$ for $\alpha\in\Phi_+^\ell$
and $k\in\mathbb{Z}$. By \eqref{shiftinvariant} if suffices to show that 
$\chi_{s,t}(y_{-\alpha+\delta})=s$ for $\alpha\in\Phi_+^\ell$. This is true for $\alpha=\theta=2\epsilon_1$
since $y_{-\theta+\delta}=y_{\alpha_0}=Y_0$. Suppose it is true for $\alpha=2\epsilon_i$ with 
$i=1,\ldots,j-1$ and $2\leq j\leq r$. Then $-2\epsilon_j=\beta+\gamma$ with 
$\beta=-\epsilon_{j-1}-\epsilon_j$ and $\gamma=\epsilon_{j-1}-\epsilon_j$, and $\beta-\gamma=
-2\epsilon_{j-1}$. Hence
\begin{equation}\label{almoststep}
\lbrack y_{\beta+\delta},y_\gamma]=\kappa_{\beta,\gamma}y_{-2\epsilon_j+\delta}-
\kappa_{\beta,-\gamma}y_{-2\epsilon_{j-1}+\delta}
\end{equation}
by \eqref{rel2}. By a direct computation, $\kappa_{\beta,\gamma}=2=\kappa_{\beta,-\gamma}$,
hence it follows from \eqref{almoststep} and the induction hypothesis that 
$\chi_{s,t}(y_{-2\epsilon_j+\delta})=s$. This shows that
$\chi_{s,t}(y_{-\alpha+\delta})=s$ for $\alpha\in\Phi_+^\ell$.

Finally, by \eqref{rel1} we have 
\[
\lbrack y_{\alpha_i},y_{\alpha_i+k\delta}\rbrack=y_{k\delta}^{(i)}
\]
for $k\in\mathbb{Z}\setminus\{0\}$ and $1\leq i\leq r$, hence $\chi_{s,t}(y_{k\delta}^{(i)})=0$.
\end{proof}

%%%%%%%%%%%%%%%%%%%%%%%%%%%%%%%%%%%%%%%%%%%%%%%%

\end{document}